\documentclass[a4paper,11pt]{amsart}
\usepackage[utf8]{inputenc}
\usepackage{amssymb, mathrsfs, amsfonts, amsmath}
\usepackage{mathtools} 
\usepackage{setspace}
\usepackage{hyperref}
\usepackage{amsthm}
\usepackage{tikz}
\usepackage[english]{babel}
\usepackage{xcolor} 
\usepackage{geometry}
\usepackage{lipsum}

\title[Oscillatory Carleson theorems]{Uniform bounds for oscillatory and polynomial Carleson operators}
\author{Jo\~ao P. G. Ramos}
\address{ETH Z\"urich, D-MATH - R\"amistrasse, 101, 8092 Z\"urich - Switzerland}
\email{joao.ramos@math.ethz.ch} 

\newcommand{\R}{\mathbb{R}}
\newcommand{\Z}{\mathbb{Z}}

\newcommand{\mmd}{\mathrm{d}}

\newtheorem{theorem}{Theorem}

\newtheorem{lemma}{Lemma}

\newcommand{\tP}{\tilde{P}}

\def\XXint#1#2#3{{\setbox0=\hbox{$#1{#2#3}{\int}$}
\vcenter{\hbox{$#2#3$}}\kern-.5\wd0}}

\begin{document}
\begin{abstract}
We prove that a variety of oscillatory and polynomial Carleson operators are uniformly bounded on the family of parameters under considerations. As a particular application of our techniques, we prove uniform bounds
for oscillatory Carleson operators near a single scale version of the quadratic Carleson operator. 
\end{abstract}
\maketitle

\section{Introduction}
For $f \in \mathcal{S}(\R),$ one defines the \emph{Carleson operator} to be 
\[
Cf(x) = \sup_{N \in \R} \left| \int_{\R} f(x-t) e^{iNt} \frac{\mmd t}{t}\right| .
\]
A celebrated result related to this operator is that it is bounded on $L^2(\R).$ This was first proved by Carleson \cite{Carleson}, although not directly, as a by-product of 
his proof of almost everywhere convergence of Fourier series of $L^2(\mathbb{T})-$functions. Subsequent works (see, for instance, \cite{Fefferman, LaceyThiele, Hunt}) have contributed to simplify Carleson's proof, connect it to different contexts
and extend it to other $L^p-$spaces. 

In this note, however, we will concentrate on a question originally raised by E. Stein and other questions derived from it: if we define the \emph{polynomial} Carleson theorem of degree $d$ to be 
\begin{equation}\label{eq poly carleson}
C_d f(x) = \sup_{\text{deg} (P) \le d} \left| \int_{\R} f(x-t) e^{i P(t)} \, \frac{\mmd t}{t} \right|,
\end{equation}
is this bounded in $L^p(\R)$ for any $d \ge 1$? For $d=1,$ this is Carleson's theorem. A first contribution in the $d \ge 2$ case was made by Stein and Wainger \cite{SteinWainger}, which proved that if one restricts the supremum in 
\eqref{eq poly carleson} to the class of polynomials of degree at most $d$ such that $P(0) = P'(0) = 0,$ then this new operator is bounded in $L^p(\R),$ for all $p >1.$ We refer the reader aditionally to \cite{GPRY, GHLR, PierceYung, Guo2} for additional developments in this regard. 

Nevertheless, it was not until the work of V. Lie \cite{Lie1} that the first major contribution in this regard was made. In that work, Lie proved that the \emph{quadratic} Carleson operator $C_2$ in the definition above possesses a weak-type bound in $L^2.$ The methods rely on appropriate 
adaptations and decompositions of the quadratic Carleson operator in the same fashion as in the work of Fefferman \cite{Fefferman}. In fact, a similar idea has been employed by both Lie \cite{Lie2} and Zorin-Kranich \cite{pavel} to conclude that, in the general case of $d \ge 2,$ the polynomial Carleson theorem $C_d$, together with 
truncated and generalized versions of it, is bounded in $L^p(\R)$ for any $p>1.$ 

Another related problem is that of bounding other kinds of \emph{oscillatory} Carleson operators, as raised in \cite{GHLR}. There, the authors, driven by the study of certain Hilbert transforms along variable curves, arrive at bounding the \emph{oscillatory} Carleson operators 
\[
C_{\alpha}f(x) = \sup_{N,M \in \R} \left| \int_{\R} f(x-t) e^{iNt} e^{iM[t]^{\alpha}} \, \frac{\mmd t}{t} \right|.
\]
Here, we let $[t]^{\alpha}$ denote either $\text{sign}(t)|t|^{\alpha}$ or $|t|^{\alpha}$ Although the authors prove in \cite{GHLR} that these operators are bounded in $L^p$ whenever $\alpha \in \R, \alpha \not\in\{0,1\},$ the proof is dramatically different for two distinct cases: for $\alpha\neq 2,$ one can 
compare the operator on a certain subset of $\R$ to a Carleson operator, and outside that region it is possible to use a $TT^*-$method, in the same fashion as Stein and Wainger \cite{SteinWainger}, to obtain $L^p-$bounds for such operator. On the other hand, for $\alpha = 2,$ the only proof known so far is that of Lie \cite{Lie1}, which 
employs time-frequency analysis methods and a strategy resembling the original Carleson theorem proof. Due to the difference in such proofs, the bounds on the constant $\|C_{\alpha}\|_{L^p \to L^p}$ \emph{blow up} as $\alpha \to 2$ in \cite{GHLR}. 

In this note, we will mainly focus on proving \emph{uniform bounds} for some instances of such operators. We start by providing an alternative proof for 

\begin{theorem}\label{poly} Let $P(t)$ be a polynomial of degree $d.$ Then the operator 
\[
C_Pf(x) = \sup_{N \in \R} \left| \int_{\R} f(x-t) e^{iNt + iP(t)} \frac{\mmd t}{t} \right| 
\]
is bounded in $L^p, \, 1< p < +\infty,$ with bounds depending \emph{only} on $p,d.$ 
\end{theorem}

Notice that, by the works of Lie \cite{Lie2} and Zorin-Kranich \cite{pavel}, this result is not new as stated, being a consequence of the more general polynomial Carleson theorem. We stress, however, that the proof we provide here is conceptually different: we prove that boundedness of the 
original (linear) Carleson operator \emph{imply} bounds for this polynomial version. This is connected to recent results by this author, which use operators similar to $C_P$ above in order to study degenerate cases of maximal modulations of the Hilbert transform along the parabola. Indeed, in \cite{Ramos1}, 
we consider the multiplier 
$$m(\xi,\eta) = \text{p.v.}\int_{\R} e^{2 \pi i\xi t + 2\pi i\eta t^2} \, \frac{\mmd t}{t},$$
which is associated to the operator $\mathcal{H}_2f(x,y) = \int_{\R} f(x-t,y-t^2)\, \frac{\mmd t}{t}$, and possesses an anisotropic dilation invariance. We thus define the family 
\[
m_{a,b}(\eta) =  m(a\eta + b, \eta), a,b \in \R.
\]
In order to prove that the operators $\mathcal{C}_{a,b}f(x) = \sup_{N \in \R} |T_{a,b}(\mathcal{M}_Nf)(x)|,$  where $T_{a,b}h = (m_{a,b}\widehat{h})^{\vee},$ are bounded in $L^p$ \emph{uniformly} on $a,b \in \R,$ one needs to resort to Theorem \ref{poly}. Indeed, in \cite{Ramos1}, we use a stronger version of Theorem \ref{poly} in order to obtain an even 
stronger statement about uniformity of bounds for $\mathcal{C}_{a,b},$ but with the present methods we are already able to answer the question of how to obtain Theorem \ref{poly} directly from bounds for the Carleson operator. 

For the proof of Theorem \ref{poly}, we have two main steps: the first is to use a \emph{gap decomposition}, in the same spirit of Guo \cite{Guo1} (see also \cite{CarberyRicciWright} for the original idea behind this decomposition), to reduce matters to the quadratic case. The second step deals is to deal with the version of
the quadratic carleson operator arising from that. Although bounds for the operator 
\[
\tilde{C}_2f(x) = \sup_{N \in \R} \left| \int_{\R} f(x-t) e^{iNt} e^{it^2} \, \frac{\mmd t}{t}\right| 
\]
follow directly from the boundedness of the Carleson operator by using the mapping $f(t) \mapsto e^{it^2}f(t),$ this also provides us with a natural connecting point with the next result.

\begin{theorem}\label{known} Let $\alpha >0, \, \alpha \not\in \{0,1\}.$ Then the oscillatory Carleson operator 
\[
\mathcal{C}_{\alpha} f(x) :=  \sup_{N \in \R} \left|\int_{\R} f(x-t) e^{iNt} e^{i[t]^{\alpha}} \frac{\mmd t}{t} \right|
\]
is bounded in $L^p(\R), 1 < p < + \infty,$ with $\sup_{\alpha \in (2-\beta,2+\beta)} \|\mathcal{C}_{\alpha}\|_{p \to p} < +\infty$ whenever $0< \beta < 1$. 
\end{theorem} 

Of course, the bounds for each of the operators $\mathcal{C}_{\alpha}$ individually is known, but, as previously stated, the point here is to prove that bounds hold uniformly in a viscinity of $\alpha = 2,$ with our proof being new and, to the best of our knowledge, it cannot be
derived in a striaghtforward manner from the results of Lie \cite{Lie1, Lie2}. 

The main new idea for the proof of Theorem \ref{known} is a deeper understanding of a wave packet decomposition guided by the oscillatory factor. Instead of redoing the original Carleson proof, we identify in which regions the best strategy is to compare our operator with a (truncated) Carleson operator, and where the oscillation 
coming from the $[t]^{\alpha}$ term takes over and induces decay. 

This paper is organized as follows: in Section \ref{simplerr}, we prove Theorem \ref{poly} by reducing to the quadratic case. We then focus, on Section \ref{uniform}, on 
proving bounds for the the case $|\alpha - 2| < \frac{1}{2}$ of Theorem \ref{known} with bounds uniform in the parameter $\alpha$ (for the remaining cases, see, for instance, \cite[Corollary~1.7]{GHLR}). We reserve the last section for related remarks and comments. 

\subsection*{Acknowledgements} The author is thankful to Christoph Thiele, Pavel Zorin--Kranich and Shaoming Guo for discussions that led to the final form of this manuscript. 
 
\section{Theorem \ref{poly} and bounds for a simplified polynomial Carleson operator}\label{simplerr} In this section, we prove that the operator 
\[
C_Pf(x) := \sup_{N \in \R} \left|\int_{\R} f(x-t) e^{iNt + iP(t)} \,\frac{\mmd t}{t}\right|,
\]
where $P$ is a fixed polynomial, is bounded in $L^p(\R), p>1,$ with bounds depending only on the degree of $P.$ We focus thus on proving bounds for 
\[
C_P^Nf(x) := \int_{\R} f(x-t) e^{iN(x)t + iP(t)} \, \frac {\mmd t}t, 
\]
in $L^p$ independently of $N$, depending only on $\text{deg }P,$ where $N : \R \to \R_{+}$ is a measurable function which we may choose as taking on only finitely many values. 

\label{redquad} We first assume bounds on the case $\text{deg }P = 2$ to obtain the general case. We follow closely the approach by Guo \cite{Guo1} in order to achieve uniform bounds on oscillatory singular integrals with respect to 
fewnomials. Namely, we find a decomposition of the real line associated to the polynomial $P$, but with parameters depending only on $\text{deg} P.$ \\

For that purpose we let $N(x)t + P(t) = N'(x)t + at^2 + \tilde{P}(t),$ where $\tP(0) = \tilde{P}'(0) = \tilde{P}''(0) = 0$. We let $a_i$ be the coefficient of $t^i$ in the expansion of $\tilde{P}(t)$, and define $\lambda_k = 2^{1/k}.$ Finally, fix $C_d > 0$ large enough -- $C_d \ge 10^{d!}$ will work, for instance -- and 
define the first set of \emph{bad scales} associated to $j,k \in [3,d]$ to be 
\[
\mathcal{S}_{bad}^1(j,k) :=\{l \in \Z \colon 2^{-C_d} |a_k \lambda_d^{kl}| \le |a_j \lambda_d^{jl}| \le 2^{C_d} |a_k \lambda_d^{kl}|\}.
\]
This set is connected and has cardinality at most $4dC_d.$ In fact, if $l_1, l_2 \in \mathcal{S}_{bad}^1(j,k)$, then the corresponding equations in the definition imply that, for $\theta \in (0,1),$ 
\[
2^{-C_d} |a_k \lambda_d^{k(l_1 \theta + (1-\theta)l_2)} | \le |a_j \lambda_d^{j(l_1 \theta + (1-\theta)l_2)}| \le 2^{C_d} |a_k \lambda_n^{k(l_1\theta + (1-\theta)l_2)}|.
\]
This plainly implies connectivity. Also, the minimal scale $l_{min} \in \mathcal{S}^1_{bad}(j,k)$ satisfies that $2^{-C_d} |a_k \lambda_d^{kl_{min}}| \sim |a_j \lambda_d^{jl_{min}}|.$ This implies 
$l_{min} = \frac{1}{j-k} \left(d\log_2(|a_k|/|a_j|) - dC_d\right).$ The same analysis shows that $l_{max} = \frac{1}{j-k} \left(d \log_2 (|a_k|/|a_j|) + d C_d\right).$ The cardinality assertion then follows. \\

We define, therefore, 
\[
\mathcal{S}_{bad}^1 = \bigcup_{j \ne k} \mathcal{S}_{bad}^1(j,k), \,\,\mathcal{S}_{good}^1 = (\mathcal{S}_{bad}^1)^c. 
\]
The set $\mathcal{S}_{good}^1$ has at most $d^2$ connected components, in each of which a monomial ``dominates" $\tP$. Indeed, we have at most $ \text{deg}(P)^2$ nonempty sets $\mathcal{S}^1_{bad}(k,j),$ all of which are intervals. Their complement is then 
a union of at most $d^2$ intervals. We let $\mathcal{S}^2_{bad}(k,j)$ be defined in complete analogy to $\mathcal{S}^1_{bad}(k,j)$, but with respect to the sequence $b_j = j(j-1)(j-2) a_j$ of coefficients of the 
third derivative of $\tP.$ As before, we take
$$\mathcal{S}_{bad}^2 = \cup_{j \ne k} \mathcal{S}_{bad}^2(j,k), \, \mathcal{S}_{good} = \mathcal{S}_{good} \backslash \mathcal{S}_{bad}^2.$$
$\mathcal{S}_{good}$ has therefore at most $2d^2$ connected components, in each of which both $\tP$ and $\tP'''$ are ``dominated" by a monomial. \\ 

We use this scale decomposition to bound 
\[
C_P^Nf(x) := \sum_{l \in \Z} \int_{\R} f(x-t) e^{iN(x)t + iP(t)} \psi_l(t) \frac{\mmd t}{t},
\]
where $\psi_0(t)$ is a smooth function supported in $[-\lambda_d^2,-\lambda_d^{-1}] \cup [\lambda_d^{-1},\lambda_d^2]$ such that 
\[
\sum_{l \in \Z} \psi_l(t) := \sum_{l \in \Z} \psi_0\left(\frac{t}{\lambda_d^l}\right) = 1, \, \forall t \ne 0.
\]
As we saw before, the set $\mathcal{S}_{bad} := \mathcal{S}^1_{bad} \cup \mathcal{S}^2_{bad}$ is finite, with cardinality bounded by a constant depending only on $d.$ On each of the scales $l \in \mathcal{S}_{bad},$ 
we have 
\[
\left|\int_{\R} f(x-t) e^{iN(x)t + iP(t)} \psi_l(t) \frac{\mmd t}{t}\right| \le 4 Mf(x),
\]
where $M$ denotes the Hardy--Littlewood maximal function. Therefore, 
\[
\left|\sum_{l \in \mathcal{S}_{bad}} \int_{\R} f(x-t) e^{iN(x)t +iP(t)} \psi_l(t) \frac{\mmd t}{t} \right| \le A_d \cdot Mf(x).
\]
Here, $A_d$ only depends on the number of bad scales, which is bounded by $(4dC_d)^2.$ This completes the treatment of $\mathcal{S}_{bad}$. For the good scales, we pick each connected component $\mathcal{S}_{good}(k_1,k_2) \subset \mathcal{S}_{good}$ associated to a pair $k_1,k_2 \in [3,d]$ so that $l \in \mathcal{S}_{good}(k_1,k_2),
\, t \in [\lambda_d^{l-2},\lambda_d^{l+1}],$
\[
\frac{1}{2} \cdot (1-d2^{-C_d})|a_{k_1} t^{k_1}| \le |\tP(t)| \le 2 (1+d2^{-C_d}) |a_{k_1} t^{k_1}|,
\]
\[
2(1+d2^{-C_d})|a_{k_2}t^{k_2-3}| \ge  |\tP^{(3)}(t)| \ge \frac{1}{2} \cdot (1-d2^{-C_d})|a_{k_2}t^{k_2-3}|.
\]
In particular, $|\tP(t)| \le 8|a_{k_1} t^{k_1}|, \, |\tP^{(3)}(t)| \ge \frac{1}{8} |a_{k_1} t^{k_1-3}|.$ Next, we gather scales in a single connected component by 
\[
\Phi_{k_1,k_2} (t) = \sum_{l \in \mathcal{S}_{good}(k_1,k_2)} \psi_l(t).
\]
Now let $\phi^k_{0}$ be a smooth positive function supported in $[-\lambda_k^2,-\lambda_k^{-1}] \cup [\lambda_k^{-1},\lambda_k^2]$ such that 
\[
\sum_{i \in \Z} \phi^k_i(t) := \sum_{i \in \Z} \phi^k_0\left(\frac{t}{\lambda_k^i}\right) = 1, \, \forall t \ne 0.
\]
Finally, if $B_k \in \Z$ is so that $\lambda_k^{-B_k} \le |a_k| \le \lambda_k^{-B_k + 1},$ we let $\theta_k = B_k/k$. \\ 

From the previous considerations, we are left to bound 
\[
\sum_{l \in \mathcal{S}_{good}} \int_{\R} f(x-t) e^{iN(x)t +iP(t)} \psi_l(t) \frac{\mmd t}{t}= \sum_{(k_1,k_2)} \int_{\R} f(x-t) e^{iN(x)t + iP(t))} \Phi_{k_1,k_2}(t) \frac{\mmd t}{t}.
\]
As the number of connected components of $\mathcal{S}_{good}$ is $\le 2d^2,$ we only need to bound each summand individually, which we write as 
\[
\sum_{i \in \Z} T^{k_1,k_2}_if(x) := \sum_{i \in \Z} \int_{\R} f(x-t) e^{iN(x)t + iP(t)} \Phi_{k_1,k_2}(t) \phi^{k_1}_i(t) \frac{\mmd t}{t}.
\]
We split the sum above further as 
\[
\sum_{i \le \theta_{k_1}} T^{k_1,k_2}_if(x) + \sum_{i > \theta_{k_1}} T^{k_1,k_2}_if(x).
\]
By comparing each term in the first summand, we arrive at 
\begin{align*}
\left|\sum_{i \le \theta_{k_1}} T^{k_1,k_2}_if(x) \right| \le & \left| \sum_{i \le \theta_{k_1}}  \int_{\R} f(x-t)  e^{iN(x)t + iat^2} \Phi_{k_1,k_2}(t) \phi^{k_1}_i(t) \frac{\mmd t}{t}\right| \cr
 & + \sum_{i \le \theta_{k_1}} \int_{\R}|f(x-t)| |\tP(t)|  \Phi_{k_1,k_2}(t) \phi^{k_1}_i(t) \frac{\mmd t}{|t|}.
\end{align*}
In the first integral, the main observation is that $\Phi_{k_1,k_2} \cdot \left(\sum_{i \le \theta_{k_1}} \phi_i^{k_1}\right)$ approximates the characteristic function of an interval. The first term on the right hand is then comparable to truncations of the quadratic Carleson operator, where the supremum is taken over phases of the form $Nt + at^2,\,$ $a$ fixed.
The remaining error terms amount to an absolute constant times a maximal function. \\

It is not difficult to show that the integrand in the second term on the right hand side is bounded by $8|f(x-t)|\lambda_{k_1}^{(i-\theta_{k_1}) \cdot k_1} \cdot \lambda_{k_1}^{-i}$ on a set of measure $\sim \lambda_{k_1}^i.$ 
This is due to the growth estimate on $|\tP(t)| \le 4 |a_{k_1} t^{k_1}|$ on the support of $\phi_i^{k_1}$. Summing in $i \le \theta_{k_1}$ yields that the summand is bounded by at most $10^2 \cdot Mf(x).$ \\ 

For the remaining part, we use the $TT^*$ method. We wish to bound
\[
\sum_{i = 0}^{\infty} \|T^{k_1,k_2}_{i + \theta_{k_1}} f\|_p.
\]
In order to do it, we notice the pointwise bound $|T^{k_1,k_2}_{i+\theta_{k_1}}f| \le 2 Mf(x).$ We now need to prove exponential decay for the $L^2$ bounds of $T^{k_1,k_2}_{i+ \theta_{k_1}}.$ This follows from
proving the same exponential decay in bounds for $T^{k_1,k_2}_{i+\theta_{k_1}}(T^{k_1,k_2}_{i+\theta_{k_1}})^*$. After a change of variables, the convolution kernel of this last expression is given by
\begin{align}\label{eq oscillatory_0}
(\lambda_{k_1})^{-(i+\theta_{k_1})} \int_{\R} & e^{i(N(y)-N(y))\cdot (\lambda_{k_1})^{i+\theta_{k_1}} \cdot s'} \cr 
& \times e^{i [\tP(\lambda_{k_1}^{i + \theta_{k_1}}s') - \tP(\lambda_{k_1}^{i + \theta_{k_1}}(s'-\xi'))] } \cdot \frac{\phi^{k_1}_0(s')}{s'} \cdot \frac{\phi^{k_1}_0(s'-\xi')}{s'-\xi'} \, \mmd s', 
\end{align}
where $\xi = (\lambda_{k_1})^{i+ \theta_{k_1}} \xi'.$ We use, as per usual in such contexts, stationary phase. Our phase function this time is 
\[
v\cdot (\lambda_{k_1})^{i+\theta_{k_1}} \cdot s' + \tP(\lambda_{k_1}^{i + \theta_{k_1}}s') - \tP(\lambda_{k_1}^{i + \theta_{k_1}}(s'-\xi')).
\]
Differentiating twice, we obtain that the second derivative of the phase in $s'$ is at least as large as 
\begin{align*}
 \lambda_{k_1}^{2(i+\theta_{k_1})} |\tP''(\lambda_{k_1}^{i + \theta_{k_1}}s') - \tP''(\lambda_{k_1}^{i + \theta_{k_1}}(s'-\xi'))| \cr 
 \ge \frac{1}{8} \cdot \lambda_{k_1}^{3(i+\theta_{k_1})} |\xi'| |a_{k_1}| |(\lambda_{k_1})^{i+\theta_{k_1}}|^{k_1 - 3} \ge 2^{i-4} |\xi'|, 
\end{align*}
where we used the mean value theorem and the lower bound on $\tP^{(3)}$ on the scale $\lambda_{k_1}^{i+\theta_{k_1}}.$ The proof is finished by splitting between $|\xi'| \ge 2^{-i/10}$ and $|\xi'| \le 2^{-i/10}.$ 

For the case $|\xi'| \le 2^{-i/10},$ we bound the integral simply by pulling the absolute value inside, which then yields that \eqref{eq oscillatory_0} is controlled by 
$ C (\lambda_{k_1})^{-(i+\theta_{k_1})} 1_{\left\{|\xi| \le 2^{-i/10} (\lambda_{k_1})^{-(i+\theta_{k_1})}\right\}}.$ For the case $|\xi'| \ge 2^{-i/10},$ the usual van der Corput estimate yields that \eqref{eq oscillatory_0} is bounded by 
$C (\lambda_{k_1})^{-(i+\theta_{k_1})} \cdot 2^{-\left(\frac{9i}{30} - 4\right)} 1_{\left\{|\xi| \le 4 \cdot (\lambda_{k_1})^{-(i+\theta_{k_1})}\right\}}.$ Using this in the definition of 
$T^{k_1,k_2}_{i+\theta_{k_1}}(T^{k_1,k_2}_{i+\theta_{k_1}})^*$ yields the pointwise bound 
\begin{align*}
|T^{k_1,k_2}_{i+\theta_{k_1}}(T^{k_1,k_2}_{i+\theta_{k_1}})^*f(x)| \lesssim (2^{-i/10} + 2^{-\frac{9i}{30}})Mf(x),
\end{align*}
which then implies the desired $L^2-$exponential decay in $i,$ concluding the proof. For more details on the method and estimates used, see, for instance, the proof of Corollary 1.7 in \cite{GHLR} or the proof of Theorem 2 in \cite{Ramos1}.  

This finishes the reduction to the quadratic case, and, by either Theorem \ref{known} or the argument sketched in the introduction, the general case follows. 

\section{Theorem \ref{known} and uniform bounds for maximally modulated oscillatory singular integrals}\label{uniform} We focus on bounding
\[
C_{\alpha}^Nf(x) := \int_{\R} f(x-t) e^{iN(x)t} e^{i|t|^{\alpha}} \,\frac{\mmd t}{t}
\]
in $L^p$ independently of $N:\R \to \R_{+}.$ First we employ a dyadic decomposition: we write, for $\varphi$ smooth, nonnegative function so that $\sum_{n \in \Z} \varphi(2^n t) = 1, \, \forall t \ge 0,$
\[
C_{\alpha}^Nf(x) = \int_{\R} f(x-t) e^{iN(x)t} e^{i|t|^{\alpha}} \varphi_0(t) \, \frac{\mmd t}{t} +  \sum_{n > 0} \int_{\R} f(x-t) e^{iN(x)t} e^{i|t|^{\alpha}} \varphi(2^{-n}t) \, \frac{\mmd t}{t}.
\]
Here, $\varphi_0(t) = \sum_{n \le 0} \varphi(2^{-n} t).$ This implies that the first summand is, modulo a maximal function error, a truncated Carleson operator. We analyze the remaining sum. \\ 

In order to capture the oscillation provided by the $|t|^{\alpha}$ factor, we do a further decomposition in each of the summands. Let 
$$\gamma_n = \frac{1}{(\alpha (2^{\alpha-1} -1))^{1/2}} 2^{(1- \frac{\alpha}{2})n}.$$
Let also $\beta_n = 2^{ n} - \frac{3}{2} \gamma_n.$ Modulo maximal function errors again, we see that $\int_{\R} f(x-t) e^{iN(x)t} e^{i|t|^{\alpha}} \varphi(2^{-n}t) \, \frac{\mmd t}{t}$ may be written as 
a sum of 
\begin{equation}\label{decalpha}
\int_{\R} f(x-t) e^{iN(x)t} e^{i|t|^{\alpha}} \chi_0\left(\frac{t- \beta_n}{\gamma_n} - j\right) \, \frac{\mmd t}{t} 
\end{equation}
where $j$ ranges from $0$ to $\sqrt{\alpha(2^{\alpha-1}-1)} 2^{\frac{\alpha}{2} \cdot n}.$ Here we choose $\chi_0$ to be a positive smooth function supported on $[-3/4,3/4]$ such that 
$$\sum_{j \in \Z} \chi_j(y) := \sum_{j \in \Z} \chi_0(y-j) = 1,\, \forall y \in \R.$$
In particular, the difference between \eqref{decalpha} and 
\begin{equation}\label{model1}
\frac{1}{(j+1)\gamma_n + \beta_n} T_{n,j,\alpha}f(x) = \frac{1}{(j+1)\gamma_n + \beta_n}\int_{\R} f(x-t) e^{iN(x)t} e^{i|t|^{\alpha}} \chi_0\left(\frac{t- \beta_n}{\gamma_n} - j\right) \, \mmd t 
\end{equation}
is bounded by a universal constant times $\frac{\gamma_n}{((j+1/2)\gamma_n + \beta_n)^2}\int_{(j+1/2)\gamma_n + \beta_n}^{(j+4)\gamma_n + \beta_n} |f(x-t)| \, \mmd t.$ As the function 
$$t \mapsto \frac{\gamma_n}{((j+1)\gamma_n + \beta_n)^2} \chi_{((j+1/2)\gamma_n + \beta_n),((j+4)\gamma_n + \beta_n)}(t)$$
is in $L^1$ with norm bounded \emph{uniformly} in $\alpha \ge \frac{3}{2},$ we thus focus on \eqref{model1}. The sum 
\[
\sum_{0 \le j \le \sqrt{\alpha(2^{\alpha-1}-1)} 2^{\frac{\alpha}{2} \cdot n}} \frac{1}{(j+1)\gamma_n + \beta_n} T_{n,j,\alpha}f 
\]
is pointwise bounded by $10 \cdot Mf,$ so we seek for decay in the $L^2$ bounds. We have
\begin{align*}
 & \left|\left\langle \sum_{0 \le j \le \sqrt{\alpha(2^{\alpha-1}-1)} 2^{\frac{\alpha}{2} \cdot n}} \frac{1}{(j+1)\gamma_n + \beta_n} T_{n,j,\alpha}f, g \right\rangle\right| \le 4 \cdot 2^{-n} \sum_{0 \le j \le 2^n \cdot \gamma_n^{-1} } |\langle T_{n,j,\alpha}f, g \rangle|.\cr 
\end{align*}
We further decompose our already localized operators in a wave packet-like manner: 
\[
T_{n,j, \alpha}f(x) = \sum_{k,l \in \Z} 1_{E^{\alpha}_{k,l}}(x)T_{n,j,\alpha}f(x) := \sum_{k,l \in \Z} T^{k,l}_{n,j,\alpha}f(x),
\]
where $E^{\alpha}_{k,l} = \{ y \in \R \colon y \in (k \cdot \gamma_n, (k+1)\cdot \gamma_n], N(y) \in (l \cdot \gamma_n^{-1}, (l+1) \cdot \gamma_n^{-1}]\}.$ This decomposition takes into account the spatial localization 
of the point $x \in \R$ as well as where the measurable function $N$ lands. We thus bound:
\begin{align*}
\sum_{0 \le j \le 2^n\cdot \gamma_n^{-1}} |\langle T_{n,j,\alpha} f,g \rangle| & \le \sum_{0 \le j \le 2^n \cdot \gamma_n^{-1}} \left\| \sum_{k,l \in \Z} (T^{k,l}_{n,j,\alpha})^{*} g\right\|_{2} \cr 
 & \le \left(\frac{2^{n}}{\gamma_n}\right)^{1/2} \left( \sum_{0 \le j \le 2^n \cdot \gamma_n^{-1}} \left\| \sum_{k,l \in \Z} (T^{k,l}_{n,j,\alpha})^{*} g\right\|_{2}^2 \right)^{1/2},
\end{align*}
using that $\|f\|_2 = 1.$ 

The next step in order to conclude is to analyse $T_{n,j,\alpha}^{k,l}$ and their mutual interactions, in order to capture the oscillatory effect of the phase.  

\begin{lemma}\label{wavepack2}Let $T_{n,j,\alpha}^{k,l}$ be defined as above. Its adjoint is then given by
\begin{equation}\label{adjointt2}
(T_{n,j,\alpha}^{k,l})^*g(y) = \int_{\R} 1_{E^{\alpha}_{k,l}}(x) \chi_0\left(\frac{(x-y)- \beta_n}{\gamma_n} - j\right) e^{iN(x)(x-y) + i|x-y|^{\alpha}} g(x)\, \mmd x.
\end{equation}
Then we have $\|T_{n,j,\alpha}^{k,l}\|_{2 \to 2} = \|(T_{n,j,\alpha}^{k,l})^*\|_{2 \to 2} \le \gamma_n,$ and 
\begin{equation}\label{orthh2}
|\langle (T_{n,j,\alpha}^{k_1,l_1})^* g_1 , (T_{n,j,\alpha}^{k_2,l_2})^* g_2 \rangle| \le C \cdot 1_{|k_1-k_2| \le 10} \gamma_n \left(1+\frac{|l_1 - l_2|}{\gamma_n}\right)^{-20} \left(\int_{E^{\alpha}_{k_1,l_1}} |g_1|\right)\cdot\left(\int_{E^{\alpha}_{k_2,l_2}} |g_2| \right),
\end{equation}
whenever $k_1,k_2,l_1,l_2 \in \Z, j \le 2^n \cdot \gamma_n^{-1}$, where the constant $C$ is universal for $\alpha$ close to $2$. 
\end{lemma}

\begin{proof}[Proof of Lemma \ref{wavepack2}] The estimate on the $L^2$ norm is again a trivial application of H\"older's inequality. The inner product estimate can be done by simply writing
\begin{equation}\label{bound2} 
|\langle (T_{n,j,\alpha}^{k_1,l_1})^* g_1 , (T_{n,j,\alpha}^{k_2,l_2})^* g_2 \rangle| \le \int \int |(g_1 1_{E^{\alpha}_{k_1,l_1}})(x_1)|  |(g_21_{E^{\alpha}_{k_2,l_2}})(x_2)| | G^{\alpha}_{n,j}(x_1,x_2) |\, \mmd x_1 \, \mmd x_2,
\end{equation}
where we define $G^{\alpha}_{n,j}(x_1,x_2)$ to be 
$$ \int_{\R} \chi_j\left( \frac{x_1 - y - \beta_n}{\gamma_n} \right) \chi_j \left( \frac{x_2- y - \beta_n}{\gamma_n} \right) e^{i(N(x_1)-N(x_2))y} e^{i(|x_1 - y|^\alpha - |x_2 - y|^{\alpha})} \, \mmd y.$$
The first observation to make is that, in order for the inner product not to be zero, we must have $|k_1 - k_2| \le 10.$ Moreover, the integral defining $G^{\alpha}_{n,j}(x_1,x_2)$ can be handled via basic oscillatory integral estimates:
it has support on an interval of length $\le 20 \cdot \gamma_n,$ and the phase function is given by 
\begin{align*}
 & \phi(y) = (N(x_1) - N(x_2))y + (|x_1-y|^{\alpha} - |x_2 - y|^{\alpha}).
\end{align*}
We analyze its derivative, which, as $x_1 - y, x_2 - y >0,$ is given by
$$\phi'(y) = (N(x_1) - N(x_2)) + \alpha (|x_1-y|^{\alpha-1} - |x_2-y|^{\alpha-1}).$$ 
A calculation shows that, for $x_1,x_2 \in ((k_1-10) \cdot \gamma_n, (k_1 + 10) \cdot \gamma_n],$ 
$$\alpha (|x_1-y|^{\alpha-1} - |x_2-y|^{\alpha-1}) \le 10^2 \alpha \cdot (\alpha -1) 2^{(\frac{\alpha}{2} - 1) n} \le 10^3 2^{(\frac{\alpha}{2}-1)n} \le 10^5 \gamma_n^{-1},$$
as $\alpha \cdot (2^{\alpha-1}-1), \, \alpha, \alpha -1$ all remain bounded for $|\alpha - 2| \le \frac{1}{2}.$ In addition to that, note the fact that for $N(x_i) \in (l_i \cdot \gamma_n^{-1}, (l_i + 1) \cdot \gamma_n^{-1}], \, i=1,2,$ it holds
$$|N(x_1) - N(x_2)| \ge 10^{-5} |l_1 - l_2| \gamma_n^{-1}. $$
Therefore, the derivative of the phase is bounded from below by  $10^{-6} |l_1 - l_2| \cdot \gamma_n^{-1}$ for large values of $|l_1 - l_2|,$ which implies along with Proposition 1 in \cite[Chapter~VIII]{Stein1} that 
$$ |G^{\alpha}_{n,j}(x_1,x_2)| \le C \cdot \gamma_n \left(1+\frac{|l_1 - l_2|}{\gamma_n}\right)^{-20}.$$
Inserting in \eqref{bound2} gives the claim by noting that $C$ did \emph{not} depend on $\alpha$ throughout the proof. 
\end{proof}

Let $\|g\|_1 = 1.$ Lemma \ref{wavepack2} gives: 
\begin{align*}
 \sum_{0 \le j \le 2^n \cdot \gamma_n^{-1}} \left\| \sum_{k,l \in \Z} (T^{k,l}_{n,j,\alpha})^{*} g\right\|_{2}^2 
\end{align*}
\begin{equation*} \le \sum_{\substack{{0 \le j \le 2^n \gamma_n^{-1}}\\{|k_1-k_2| \le 10}\\{ |l_1 - l_2| \le \gamma_n \cdot 2^{n/10}}}} |\langle (T_{n,j,\alpha}^{k_1,l_1})^*g, (T_{n,j,\alpha}^{k_2,l_2})^*g \rangle|  
+ \sum_{\substack{{0 \le j \le 2^n \gamma_n^{-1}}\\{|k_1-k_2| \le 10}\\{ |l_1 - l_2| > \gamma_n \cdot 2^{n/10}}}} |\langle (T_{n,j,\alpha}^{k_1,l_1})^*g, (T_{n,j,\alpha}^{k_2,l_2})^*g \rangle|  
\end{equation*}
 \[
 \le 10^2 \gamma_n \cdot 2^{n/10} \left(\sum_{\substack{{0 \le j \le 2^n \gamma_n^{-1}}\\{l,k \in \Z}}} \|(T^{k,l}_{n,j,\alpha})^*g\|_2^2 \right) + C \cdot 10^2 \cdot 2^{-n} \cdot \gamma_n.
 \]

The second summand on the right hand side contributes in the end to an universal constant times $\sum_{n \ge 0} 2^{-n} \cdot \left(\frac{2^n}{\gamma_n}\right)^{1/2} \cdot 2^{-n/2} \gamma_n^{1/2} \le 2.$ On the other hand, for the first summand we notice again that 
$(T_{n,j,\alpha}^{k,l})^*g = (T_{n,j,\alpha}^{k,l})^*(1_{N \in [l \cdot \gamma_n^{-1},(l+1) \cdot \gamma_n^{-1})}g)$ and bound:
\begin{align*}
& \sum_{\substack{{0 \le j \le 2^n \gamma_n^{-1}}\\{l,k \in \Z}}} \|(T^{k,l}_{n,j,\alpha})^*g\|_2^2 \le \sum_{l \in \Z}  \|g 1_{N \in (l \cdot \gamma_n^{-1}, (l+1)\cdot \gamma_n^{-1}]}  \|_2  \left\|\sum_{\substack{{k \in \Z}\\{0 \le j \le 2^n \cdot \gamma_n^{-1} }}} T_{n,j,\alpha}^{k,l} ( T_{n,j,\alpha}^{k,l})^* g \right\|_2 \cr
& \le \|g\|_2 \left(\sum_{l \in \Z} \left\|\sum_{\substack{{k \in \Z}\\{0 \le j \le 2^n \cdot \gamma_n^{-1} }}} T_{n,j,\alpha}^{k,l} ( T_{n,j,\alpha}^{k,l})^* g \right\|_2^2 \right)^{1/2} \cr 
\end{align*}
\begin{align*}
\le 10^3 \|g\|_2 \left(\sum_{\substack{{k,l \in \Z}\\{0 \le j \le 2^n \cdot \gamma_n^{-1} }}} \| T_{n,j,\alpha}^{k,l} ( T_{n,j,\alpha}^{k,l})^* g\|_2^2 \right)^{1/2} \le 10^3 & \gamma_n \left(\sum_{\substack{{k \in \Z}\\{0 \le j \le 2^n \cdot \gamma_n^{-1} }}} \|( T_{n,j,\alpha}^{k,l})^* g\|_2^2 \right)^{1/2}. \cr
\end{align*}
Here, we have used the fact that $\langle T^{k_1,l}_{n,j_1,\alpha} h_1 , T^{k_2,l}_{n,j_2,\alpha} h_2 \rangle = 0$ if $k_1 \ne k_2, |j_1 - j_2| > 8,$ the $L^2$ bound for each
of the $T^{k,l}_{n,j,\alpha}$ and the fact that $\sum_{l \in \Z}  \|g 1_{N \in (l \cdot \gamma_n^{-1}, (l+1)\cdot \gamma_n^{-1}]} \|_2^2 = \|g\|_2^2 =1.$ In the end, we obtain 
\[
\left( \sum_{\substack{{0 \le j \le 2^n \gamma_n^{-1}}\\{l,k \in \Z}}} \|(T^{k,l}_{n,j,\alpha})^*g\|_2^2 \right)^{1/2} \le 10^3 \gamma_n. 
\]
This implies that each summand in $n$ amounts to a contribution of at most an absolute constant independent of $\alpha$ times $2^{-\frac{9n}{20}} \gamma_n.$ The sum in $n$ of the last bounds converges for $\alpha \ge \frac{3}{2}$ to a constant bounded uniformly in $\alpha$, concluding the proof. 

\section{Comments and remarks} 

\subsection{Uniform bounds for maximally modulated oscillatory singular integrals.} Although Section \ref{uniform} deals with uniformity of $L^p(\R)$ bounds for the $\alpha \sim 2$ case, there are two other natural cases to investigate. Namely, it is conjectured that the $L^p$ constants for $C_{\alpha}$ 
remain bounded as $\alpha \to 0,$ where we recover the case of the Carleson operator. If $\alpha \to 1,$ Guo \cite{Guo2} observes that the operator without the supremum in the modulation already fails to be bounded in $L^p.$ \\ 

Analogously, if we consider the operators 
\[
C^{odd}_{\alpha} f(x) := \sup_N \left| \int_{\R} f(x-t) e^{iNt} e^{i \text{sign}(t) |t|^{\alpha}} \, \frac{\mmd t}{t} \right|,
\]
then the same proof of Section \ref{uniform} applies to prove uniform bounds near $\alpha = 2.$ For this operator, there is also an additional possible uniform bound. If $\alpha \to 1,$ then it \emph{formally} holds that
$C^{odd}_{\alpha} f \to Cf$, the Carleson operator as defined in the introduction. The proofs in \cite{GHLR} do not provide uniform bounds in either of the cases above. We expect these bounds to hold, but cannot present a proof at the moment. 

\subsection{Bounds for Stein-Wainger-type operators} Besides the operator $C^{odd}_{\alpha}$ considered before, a natural question to those who analyse the techniques above carefully is of \emph{uniform bounds} for Stein--Wainger type operators. Indeed, Guo \cite{Guo2} was the first to consider such a question, proving that the operators 
\[
\mathfrak{C}_{\alpha}f(x) = \sup_{N \in \R} \left| \int_{\R} f(x-t) e^{iN|t|^{\alpha} \text{sign}(t)} \, \frac{\mmd t}{t} \right| 
\]
are bounded in $L^p(\R),$ for any $\alpha > 0.$ Interestingly, there is a dichotomy in the techniques of proof for bounding such operators: for $\alpha = 1,$ one actually results to the celebrated Carleson-Hunt result to derive the conclusion, but for any $\alpha \neq 1,$ the approach by Guo uses a much more direct strategy: one compares the operator in a small neighbourhood of the origin to a maximally truncated Hilbert transform, and estimates the difference by an usual 
Hardy--Littlewood maximal function. Away from that neighbourhood, the strategy is to use a $TT^*$ method to obtain decay from the oscillatory nature of the phase. 

As Guo's bounds on $\|\mathfrak{C}_{\alpha}\|_{L^p \to L^p}$ blow up as $\alpha \to 1,$ one interesting question is whether elements of both proofs can be combined in order to make the bounds on $\|\mathfrak{C}_{\alpha}\|_{L^p \to L^p}$ uniform as $\alpha \to 1.$ 

We believe that our techniques in this manuscript can help shed some light on this question. In particular, we believe that a similar decomposition as that employed in the proof of Theorem \ref{known} might be useful to obtain such uniform bounds, at least in a Walsh version. We leave a more detailed discussion on this matter to a future manuscript.


\begin{thebibliography}{99}
\bibitem{CarberyRicciWright}
\newblock  A. Carbery, F. Ricci and J. Wright, 
\newblock \emph{Maximal functions and Hilbert transforms associated to polynomials.} 
\newblock Rev. Mat. Iberoam., {\bf 14} (1998), n. 1, 117--144.

\bibitem{Carleson}
\newblock L. Carleson, 
\newblock \emph{On convergence and growth of partial sums of Fourier series.}
\newblock Acta Math., {\bf 116} (1966), 135--157.

\bibitem{MChrist} 
\newblock M. Christ, 
\newblock \emph{Hilbert transforms along curves: I. Nilpotent Groups.}
\newblock  Ann. Math., {\bf 122} (1985), n. 3, 575--596.

\bibitem{ChristStein}
\newblock M. Christ, E. Stein, 
\newblock \emph{A remark on singular Calderón-Zygmund theory.} 
\newblock  Proc. Amer. Math. Soc., {\bf 99} (1987), n. 1, 71--75.

\bibitem{FabesRiviere}
\newblock E. Fabes, N. Rivi\'ere, 
\newblock \emph{Singular integrals with mixed homogeneity.}
\newblock Studia Math., {\bf 27} (1966), n. 1, 19--38. 

\bibitem{Fefferman}
\newblock C. Fefferman, 
\newblock \emph{Pointwise convergence of Fourier series.}
\newblock Ann. Math., {\bf 98} (1973), n. 3, 551--571. 

\bibitem{grafakosmodern} 
\newblock L. Grafakos, 
\newblock \emph{Modern Fourier Analysis, 3rd edition.}
\newblock Springer-Verlag New York (2014). 

\bibitem{Guo1}
\newblock S. Guo,
\newblock \emph{A remark on oscillatory integrals associated with fewnomials.}
\newblock New York J. Math., {\bf 23} (2017), 1733--1738.

\bibitem{Guo2} 
\newblock S. Guo, 
\newblock \emph{Oscillatory integrals related to Carleson's theorem: fractional monomials.}
\newblock Comm. Pure Appl. Anal., {\bf 15} (2016), no. 3, 929--946.

\bibitem{GPRY}
\newblock S. Guo, L. Pierce, J. Roos and P.L. Yung, 
\newblock \emph{Polynomial Carleson operators along monomial curves in the plane.}
\newblock J. Geom. Anal., {\bf 27} (2017), n. 4, 2977--3012. 

\bibitem{GHLR}
\newblock S. Guo, J. Hickman, V. Lie and J. Roos,
\newblock \emph{Maximal operators and Hilbert transforms along variable non-flat homogeneous curves.}
\newblock Proc. London Math. Soc., {\bf 115} (2017), n. 1, 177--219.

\bibitem{Hunt}
\newblock R. Hunt, 
\newblock \emph{On the convergence of Fourier series.}
\newblock Orthogonal Expansions and their Continuous Analogues (Proceedings of Conference, Edwardsville,Illinois, 1967) (1968), 235--255.

\bibitem{LaceyThiele}
\newblock M. Lacey, C. Thiele, 
\newblock \emph{A proof of boundedness of the Carleson operator.} 
\newblock Math. Res. Lett., {\bf 7} (2000), 361--370. 

\bibitem{Lie1} 
\newblock V. Lie, 
\newblock \emph{The (weak-$L^2$) boundedness of the quadratic Carleson operator.}
\newblock Geom. Funct. Anal., {\bf 19} (2009), n. 2, 457--497. 

\bibitem{Lie2} 
\newblock V. Lie, 
\newblock \emph{The polynomial Carleson operator.}
\newblock Ann. Math., {\bf 192} (2020), n. 1, 47--163. 

\bibitem{NRW1} 
\newblock A. Nagel, N. Rivi\'ere, S. Wainger,
\newblock \emph{On Hilbert transforms along curves.}
\newblock     Bull. Amer. Math. Soc., {\bf 80} (1974), n. 1, 106--108.

\bibitem{NRW2} 
\newblock A. Nagel, N. Rivi\'ere, S. Wainger,
\newblock \emph{On Hilbert transforms along Curves II.}
\newblock  Amer. J. Math., {\bf 98} (1976), n. 2, 395--403. 

\bibitem{PierceYung} 
\newblock L. Pierce, P.L. Yung, 
\newblock \emph{A polynomial Carleson operator along the paraboloid.}
\newblock Rev. Math. Iberoam., {\bf 35} (2019), n. 2, 339--422. 

\bibitem{Ramos1}
\newblock J.P.G. Ramos, 
\newblock \emph{The Hilbert transform along the parabola, the polynomial Carleson theorem, and oscillatory singular integrals.}
\newblock to appear at Math. Ann.; online version available at https://doi.org/10.1007/s00208-020-02075-5. 

\bibitem{Roos}
\newblock J. Roos, 
\newblock \emph{Bounds for anisotropic Carleson operators.}
\newblock J. Fourier Anal. Appl., Online version (10-Dec-2018), 1--32. 

\bibitem{SeegerTaoWright}
\newblock A. Seeger, T. Tao, J. Wright, 
\newblock \emph{Singular maximal functions and Radon transforms near $L^1$.}
\newblock Amer. J. Math., {\bf 126} (2004), n. 3, 607--647. 

\bibitem{Stein1}
\newblock E. Stein, 
\newblock \emph{Harmonic Analysis: Real variable methods, Orthogonality, and Oscilatory integrals,} 
\newblock Princeton University Press, Princeton, NJ, 1993. 

\bibitem{Stein2}
\newblock E. Stein, 
\newblock \emph{Oscillatory integrals related to Radon-like transforms.}
\newblock In Proceedings of the Conference in Honor of Jean-Pierre Kahane, Orsay, {\bf Special Issue} (1993), pp. 535--551.

\bibitem{SteinWainger}
\newblock E. Stein, S. Wainger, 
\newblock \emph{Oscillatory integrals related to Carleson's theorem.} 
\newblock Math. Res. Lett., {\bf 8} (2001), 789--800. 

\bibitem{pavel}
\newblock P. Zorin-Kranich, 
\newblock \emph{Maximal polynomial modulations of singular integrals.}
\newblock preprint available at \href{https://arxiv.org/abs/1711.03524}{\texttt{arXiv:1711.03524v5}}


\end{thebibliography}
\end{document}